\documentclass[10pt]{amsart}
\usepackage{enumerate}
\usepackage{amssymb}
\usepackage{amsfonts}
\usepackage{hyperref}
\usepackage{mathrsfs}
\theoremstyle{plain}
\newtheorem{theorem}{Theorem}[section]
\newtheorem{proposition}[theorem]{Proposition}
\newtheorem{corollary}[theorem]{Corollary}
\newtheorem{lemma}[theorem]{Lemma}

\theoremstyle{definition}
\newtheorem{definition}[theorem]{Definition}

\newtheorem*{ack}{Acknowledgement}

\theoremstyle{remark}
\newtheorem*{remark*}{Remark}
\newtheorem{remark}[theorem]{Remark}
\newtheorem{remarks}[theorem]{Remarks}

\newcommand\numberthis{\addtocounter{equation}{1}\tag{\theequation}}
\numberwithin{equation}{section}
\def\1{\textbf{\textnormal{1}}}
\def\E{{\textnormal E}}
\def\Pr{{\textnormal P^*}}
\def\P{{\textnormal P}}
\def\CLT{\textnormal{CLT}}
\def\supl{\sup\limits}
\def\R{{\mathbb R}}
\def\N{{\mathbb N}}

\begin{document}
\title{A CLT for weighted time-dependent uniform empirical processes}
\author{Yuping Yang}
\address{School of Mathematics and Statistics, Southwest University, Chongqing 400715, People's Republic of China}
\email{yangyuping@gmail.com}
\begin{abstract}
  For a uniform process $\{ X_t: t\in E\}$ (by which $X_t $ is uniformly distributed on $(0,1)$ for $t\in E$)
  and a function $w(x)>0$ on $(0,1)$, we give a sufficient
  condition for the weak convergence of the empirical
  process based on $\{ w(x)(\1_{X_t\leq x} -x): t\in E, x\in
  [0,1]\}$ in $\ell^\infty(E\times [0,1])$. When
  specializing to $w(x)\equiv 1$ and assuming strict monotonicity
  on the marginal distribution functions of the input process,
  we recover a result of \cite{KKZ10}. In the last section,
  we give an example of the main theorem.
\end{abstract}
\maketitle

\section{Introduction}

Given a sequence of independent uniform $(0,1)$ random
variables $X_1, X_2, \cdots$, if let $G_n(x)= n^{-1/2}
\sum_{i=1}^n(\1_{X_i\leq x} -x)$ be the uniform empirical
process, then Donsker's theorem (\cite{Don52}) says $G_n(x)$
converges weakly to the Brownian bridge process, $B(x) $, on
$[0,1]$. Weighted empirical processes consider suitable
weight functions $w(x)$ such that $w(x)G_n(x)$ converges
weakly to the weighted Brownian bridge process $w(x)B(x)$;
in the literature, such a theorem is called
 Chibisov-O'Reilly theorem; see \cite{Cib64}, \cite{Ore74},
\cite{CCHM86a} etc. \cite{KKZ10} considered a time dependent
empirical process
\[
 G_n(t,y):=n^{-1/2}\sum_1^{n}(\1_{Y_i(t)\leq y} - \P(Y_i(t)\leq y)), 
\quad t\in E, \quad y\in \R,
\]
for independent and identically distributed (iid) stochastic
processes \\
$Y_1(t), Y_2(t), \cdots$ for $t\in
E$. Under a condition the authors call the L-condition, this empirical process
converges weakly in $\ell^\infty(E\times \R)$.
In \cite{E-L06}, the authors proved a CLT for weighted tail
empirical processes under a small oscillation condition as the
L-condition guarantees.

We consider a time dependent weighted uniform empirical process.
For a process $X(t)$ for $t\in E$ and a ``weight function''
$w(x)$ on $(0,1)$, we are interested in conditions on the
process and the weight function so that the empirical
process
 \[
 \nu_n(t,y):=n^{-1/2}\sum_1^{n}w(y)(\1_{X_i(t)\leq y} - y)), 
\quad t\in E, \quad y\in [0,1],
\]
where $X(t), X_1(t), X_2(t),\cdots$ are iid, converges
weakly in $\ell^\infty(E\times [0,1])$. We give a sufficient
condition in Section~\ref{s:main} for a Central limit theorem (CLT)
for this empirical process.

This paper is organized as following. In
Section~\ref{s:pre}, we give some definitions and results
about weak convergence (CLT) for empirical processes.
Section~\ref{s:main} contains the main result. The proof is
to use Theorem 4.4 in \cite{AGOZ}.  In particular, the
pre-Gaussian condition and the local modulus condition are
to be checked under the assumptions. An example of the main
theorem is given at the last section.

\section{Preliminaries}\label{s:pre}
Given a centered stochastic process $\{X(t): t\in T\}$,
we define the empirical process based on it by
\begin{equation}\label{def:nu_n}
 \nu_n(t):=n^{-1/2}\sum_{j=1}^n X_j(t), \quad t\in T,
\end{equation}
where $\{X_j(t):t\in T\}$ for $j=1,2,\cdots$ are independent
and identically distributed as $\{X(t):t\in T\} $.

On a probability space $(\Omega, \mathcal{A}, P)$, recall
the outer expectation of an arbitrary function $f:\Omega
\rightarrow \R $
\[
\E^*(f):=\inf \{\, \E g: g\geq f, \, g \text{ is }
(\mathcal{A}, \mathcal{B}(\R)) \text{ measurable}\,\}.
\]
\begin{definition}
  Let $X:=\{X(t): t\in T\}$ be a centered stochastic process
  on a parameter set $T$, and sample paths in $\ell^\infty(T)$. Assume $\E|X(t)|^2<\infty$ for
  $t\in T$. The empirical process based on $X$,
  $\nu_n(t)$ in \eqref{def:nu_n}, satisfies
  the central limit theorem, -- for short $X\in
  \textup{CLT}$ -- if there exists a centered Radon measure
  $\gamma$ on $\ell^\infty(T)$ such that for all
  $H:\ell^{\infty}(T) \rightarrow \R$ bounded and
  continuous, we have
$$
\lim_{n\to \infty} \E^*(H(\nu_n)) =  \int_{\ell^\infty(T)} H\,d\gamma.
$$
\end{definition}

\begin{definition}\label{def:pregaussian}
  A centered stochastic process $\{X_t: t\in T\}$ is
  pregaussian if its covariance coincides with the
  covariance of a centered Gaussian process $G$ on $T$ with
  bounded and uniformly $d_G$-continuous sample paths, where
  $d_G(s,t):=(\E(G(s) - G(t))^2)^{1/2}$.
\end{definition}
\begin{theorem}[cf. \cite{KKZ10}, Proposition 1]\label{thm:comp}
  Let $H_1$ and $H_2$ be zero mean Gaussian processes with $L_2$
  distances $d_{H_1}$, $d_{H_2}$, respectively, on $T$. Furthermore,
  assume $T$ is countable, and $d_{H_1}(s, t) ²\leq d_{H_2}(s, t)$ for
  all $s, t \in T$. Then, $H_2$ sample bounded and uniformly
  continuous on $(T, d_{H_2})$ with probability one, implies $H_1$ is
  sample bounded and uniformly continuous on $(T, d_{H_1})$ with
  probability one.
\end{theorem}
When $T=[0,1]$, this is Lemma~2.1 in \cite{MS72}.

The assumption that $T$ is countable can be removed if 
$T$ is given a totally bounded metric.
\begin{lemma}\label{lem:extension}
  Let $\{G(t):t\in T\}$ be a zero mean Gaussian
  process. Further assume $\sup_{t\in T} \E G(t)^2 <
  \infty$.  Let $d_G(s,t):=(\E (G(s)-G(t))^2)^{1/2}$.  Then,
  if $T_0$ is a dense set in $(T, d_G)$ and the restricted
  process $\{G(t):t\in T_0\}$ is sample bounded and
  uniformly $d_G$-continuous, then $\{G(t):t\in T\}$ has a
  version with bounded and uniformly $d_G$-continuous sample
  paths.
\end{lemma}
The proof of this lemma is given in the appendix.

We will use the following theorem to prove our main result.
\begin{theorem}[\cite{AGOZ}, Theorem 4.4]\label{thm:agoz}
  Let $\{X(t): t\in T\}$ be a sample bounded process on
  a set $T$ such that $E X(t)=0$ and $ E X(t)^2< \infty$ for
  all $t\in T$. Assume:
  \begin{enumerate}[\textup{(}$i$\textup{)}]
\item $u^2P^*\{ \|X\|_\infty >u\} \rightarrow 0$ as $u\rightarrow \infty$,
 \item $X$ is pregaussian, and
 \item there is pseudometric $\rho$ on $T$ dominated by the
   pseudometric $d_G$ corresponding to a centered Gaussian
   process $G$ on $T$ with bounded and uniformly
   $d_G$-continuous paths such that for some $K$ and for all
   $t\in T$ and $\varepsilon>0$,
 $$ 
 \sup_{u>0}u^2P^*(\sup_{s\in B_{\rho}(t,\varepsilon)} |X(t)
 - X(s)|>u)\leq K\varepsilon^2.$$
  \end{enumerate}
Then $X\in \textup{CLT} $ as a $\ell^\infty(T)$-valued random
element.
\end{theorem}

\begin{definition} \label{def:dist-trans} Let $F(x)$ be a
  distribution function (df) on $\R$. The (randomized)
  distributional transform of $F(x)$ as defined in \cite{Rus09} is
  \[ \tilde F(x) := \tilde F(x, V):= F(x-) + (F(x)-F(x-))V,
  \] where $V$ is a uniform random variable on $[0,1]$.
\end{definition}

Next we give some simple properties of the distributional
transform.
\begin{lemma}\label{lem:properites-tilde-F}
\begin{enumerate}[\textup{(}$i$\textup{)}]
\item $\tilde F(x) \leq F(x)$ for all $x\in \R$.
\item If $x< y$, then $F(x) \leq \tilde F(y) $.
\item If $x\leq y$, then $\tilde F(x) \leq \tilde F(y) $.
\item If $x< y$ and $F(\cdot)$ is strictly increasing, 
then $F(x) < \tilde F(y) $.
\end{enumerate}
\end{lemma}
\begin{proof}
  By definition, $(i)$ is obvious.  For $(ii)$, take
  $x<z<y$, hence $F(x)\leq F(z)$. Since $F(z)\leq F(y-)$ and
  $F(y-)\leq \tilde F(y)$, hence $F(x) \leq \tilde F(y)$.
  For $(iii)$, if $x=y$, there is nothing to prove; assume
  $x<y$. By $(i)$ and $(ii)$, we get $(iii)$.  For $(iv)$,
  take $x<z<y$. Since $F(\cdot)$ is strictly increasing,
  $F(x)<F(z)$. But by $(ii)$, $F(z)\leq \tilde F(y)$. Hence
  $F(x)<\tilde F(y)$.
\end{proof}

For a continuous df $F$ of a random variable $X$, the random
variable $F(X)$ is uniform on $[0,1]$; but for a general df
$F$, this might not be the case. However using the
(randomized) distributional transform overcomes this.
\begin{lemma} \label{lem:dist-trans} If $F(x)$ is the
  distribution function of a random variable $X$, then
  $\tilde F(X):=\tilde F(X, V)$ is uniform on $[0,1]$.  Here $V$ is a
  uniform random variable on $[0,1] $ independent of $X$.
\end{lemma}
\begin{proof}
For a proof, see \cite{Rus09}.
\end{proof}
\begin{definition}
  We say a (pseudo) distance $\rho$ on a set $T$ is a
  continuous Gaussian distance if there is a zero mean
  Gaussian process $\{G(t): t\in T\}$ with bounded and
  uniformly $d_G$-continuous sample paths where
  $d_G(s,t):=(\E(G(s) - G(t))^2)^{1/2}$ and $\rho(s,t)=d_G(s,t)$ for all 
$s,t \in T$.
\end{definition}

For notation, we write $X$ for a process $\{X(t): t\in E\}$
and $X_t$ for $X(t)$. We recall from \cite{KKZ10}
\begin{definition}[L-condition for a stochastic process]
  \label{L-cond}
  Let $X:=\{X_t: t\in E\}$ be a stochastic process. The
  process $X$ satisfies the L-condition if there exists a
  continuous Gaussian distance $\rho$ on $E$ such that for
  every $\varepsilon>0$
\begin{equation}\label{eq:L} \sup\limits_{t\in E}
{\P}^*(\supl_{s:\rho(s,t)\leq \varepsilon} |\tilde F_t(X_t) -
\tilde F_t(X_s)| >\varepsilon^2)\leq L \varepsilon^2,
\end{equation} where $\tilde F_t(\cdot)$ is the
distributional transform of the distribution function 
$F_t(\cdot)$ of $X_t$.
\end{definition}

\begin{theorem}[\cite{KKZ10}, Theorem 3]\label{thm:kkz}
  Let $X(t)$ be a process on $E$. Let $\rho$ be given by
  $\rho(s,t)^2=\E (H(s) - H(t))^2$, for some centered
  Gaussian process $H$ that is sample bounded and uniformly
  continuous on $(E,\rho)$ with probability one. Further,
  assume that for some $L<\infty$, and all $\varepsilon>0$,
  the L-condition holds for $X$, and $D(E)$ is a collection
  of real valued functions on $E$ such that
  $\textup{P}(X(\cdot)\in D(E))=1$. If
\[
 \mathcal{C}=\{C_{s,x}:s\in E, x\in \R\},
\]
where
\[
 C_{s,x}=\{z\in D(E):z(s)\leq x\}
\]
for $s\in E$, $x\in \R$, then $\mathcal{C}\in \textup{CLT}(X)$.
\end{theorem}

In this case,
we say the empirical process based on
$\{\1_{Y(t)\leq y} - \P(Y(t)\leq y):\, t\in E,\, y\in \R\} $
satisfies the CLT or write
$\mathcal{C}\in \textup{CLT}(X)$ in $\ell^\infty(E\times \R)$.

\section{Weak convergence of the time dependent weighted
  empirical process}
\label{s:main}
In view of Theorem~\ref{thm:kkz}
and the classical weighted empirical process,
a natural question is to consider
the time dependent weighted (uniform) empirical process,
\[
\alpha_n(t,y):=n^{-1/2}\sum_{i\leq n} w(y)(\1_{X_i(t)\leq y}
-y), t\in E, \, y\in [0,1]
\]
where $\{X(t), X_1(t), X_2(t),\cdots\}$ are iid uniform
processes (see the definition below). 
Under the WL-condition (below) and some regularity conditions
on the weight function $w(\cdot)$, we proves
a CLT 
for the empirical process $\alpha_n$. 

\begin{definition}\label{def:uniform-process}
  We call a process $X=\{X(t):t\in E\}$ a uniform
  process if for each $t\in E$, $X(t)$ is uniformly
  distributed on $(0,1)$.
\end{definition}
We call the main condition in our theorem the WL-condition.
\begin{definition}\label{def:WL-condition}[WL-condition for $(X;w)$]
  Given a uniform process $X:=\{X_t: t\in E\}$ and a
  function $w:=w(x)>0$ on $(0,1)$, we say $(X;w)$ satisfies
  the WL-condition if for some constant $L$ (depending on
  $w$, but not on $x$), some continuous Gaussian distance
  $\rho$ on $E$ and all $\varepsilon>0$, $0<x<1$, we have
  \begin{align}
        \sup_t\Pr(\sup_{s:\rho(s,t)\leq \varepsilon} \1_{X_t\leq x <
      X_s}>0) &\leq \tfrac{L\varepsilon^2}{w(x)^2} \label{WLts}\\
    \sup_t\Pr(\sup_{s:\rho(s,t)\leq \varepsilon}
    \1_{X_s\leq x < X_t}>0) &\leq {L\varepsilon^2\over w(x)^2} \label{WLst}
  \end{align}
\end{definition}
The following is the main result of this paper. 
\begin{theorem}\label{thm:main1} Let $X:=\{X_t:t\in E\}$ be
  a uniform process on a parameter set $E$. Let $w:=w(x)>0, 0<x<1$ be
  continuous and symmetric about $x=1/2$ for which
  there exists $\gamma\in (0,1/2]$ such that $w$ is
  non-increasing and $xw(x)^2$ is non-decreasing on
  $(0,\gamma)$ and such that $w$ is uniformly bounded on
  $[\gamma,1/2]$. Further, assume that $w(x)$ is regularly
  varying in a neighborhood of zero and satisfies the
  integral condition
\begin{equation}\label{eq:1d-weight}
  \int_0^{\gamma} s^{-1} \exp[-c/(sw(s)^2)]\, ds <\infty \text{ for all } c>0.
\end{equation}
If
\[\lim\limits_{\alpha\to \infty} \alpha^2\Pr(\sup_{
  t\in E} w(X_t)>\alpha) = 0\] and the \textup{WL-condition}
 for $(X;w)$ is satisfied, then the empirical
  process based on $\{w(x)(\1_{X_t\leq x}-x): t\in E, \,
  x\in [0,1]\}$ converges weakly in $\ell^\infty(E\times
  [0,1])$.
\end{theorem}

\begin{remark}\label{rem:main1}
  (1) We require that the function $w(x)$ be symmetric about
  $1/2$ is no loss of generality. As the Brownian bridge has
  the same behavior at $0$ and $1$. Moreover we only give
  the proof of the theorem for $0<x<1/2$. Indeed, if let
  $\tilde X_t:=1-X_t$, then $(\tilde X;w) $ satisfies the
  WL-condition. The result for $\tilde X$ for $0<x\leq 1/2$
  gives a result of $X$ for $1/2<x\leq 1$.
  The fact (cf. \cite{GZ86lec}, Corollary 1.6, p. 61)
  that if
  $\mathcal{F}_1$ and $\mathcal{F}_2$ are Donsker classes,
  then $\mathcal{F}:=\mathcal{F}_1\cup \mathcal{F}_2$ 
 is a Donsker class gives the result for $\mathcal{F}=E\times [0,1]$.

 (2) For a general process $Y:=\{Y_t:t\in E\}$, if we define
  $X:=X_t:=\tilde F_t(Y_t)$, where $\tilde F_t(\cdot)$ is
  the (randomized) distributional transform of the df $F_t$
  of $Y_t$, then $X $ is a uniform process (see
  Lemma~\ref{lem:dist-trans}). Such a process $X$ is called
  a copula process. If we have a CLT for the $X$ process,
  then we have a CLT for the $Y $ process; see
  Proposition~\ref{prop:X2Y} for precise statement. In case
  of $w\equiv 1$, this theorem gives a proof of Theorem
  \ref{thm:kkz} provided that $F_t(\cdot)$ for
  each $t\in E $ is strictly increasing; see
  Corollary~\ref{cor:kkz}.

(3) The integral condition \eqref{eq:1d-weight} is necessary and
sufficient for one dimensional weighted uniform empirical process
under regularity of the weight function; see \cite{AGOZ}, Example 4.9.
\end{remark}

The proof of the theorem is given at the end of this
section.

The following is a possible way that a CLT for the time
dependent empirical process for $Y$ can be obtained from
proving a CLT for the process $X$.

\begin{proposition}\label{prop:X2Y} Let $w(x)$ be any function on
  $(0,1)$. Let $\{Y_t:t\in E\} $ be a process and
  $F_t(\cdot)$ is the df of $Y_t$. Let $X_t:=\tilde F_t(Y_t)$.
Then the following hold\textup{:}
\begin{enumerate}[\textup{(}i\textup{)}]
\item If $F_t(\cdot)$ is strictly increasing for
  each $t\in E $, then
 $$\{w(x)(\1_{X_t\leq x} - x) :t\in E, \, x\in [0,1]\}\in \CLT \text{ in } \ell^\infty(E\times [0,1])$$
implies
\[ \{w(F_t(y))(\1_{Y_t\leq y} - F_t(y)): t\in E, \, y\in \R \}
\in \CLT
 \text{ in } \ell^\infty(E\times \R). \]
\item Without assuming that $F_t(\cdot)$ is strictly increasing
  for each $t\in E$, we have
  $$\{w(x)(\1_{X_t\leq x} - x) :t\in E, \, x\in
  [0,1]\}\in \CLT  \text{ in } \ell^\infty(E\times [0,1])$$
  implies
\[ \{w(F_t(y))(\1_{Y_t\leq y} - F_t(y)): (t,y)\in T_0 \}
\in \CLT \text{ in } \ell^\infty(T_0),
\]
where $T_0$ is any countable subset of $E\times \R$.
\end{enumerate}
\end{proposition}
\begin{proof}
  {Proof of $(i)$}.  Recall that $\tilde F(x)\leq
  \tilde F(y) $ for $x\leq y$ and $\tilde F(x)\leq F(x)$ for
  all $x\in \R$ and for any df $F$ (see
  Lemma~\ref{lem:properites-tilde-F}). Hence $Y_t\leq y $
  implies that $\tilde F_t(Y_t) \leq F_t(y)$; i.e.
\begin{equation}\label{eq:right}
\1_{Y_t\leq y}\leq \1_{\tilde F_t(Y_t) \leq F_t(y)}, 
\text{ uniformly in } t\in E, \, y\in \R.
\end{equation}
Since $F_t(\cdot)$ is strictly increasing, by the same lemma
if $x<y$, then $F(x) < \tilde F(y) $. Now if $\tilde
F_t(Y_t) \leq F_t(y)$ and $Y_t> y $ for some $t\in E$ and
$y\in \R $, then $F_t(y)< \tilde F_t(Y_t)$.  We have a
contradiction: $F_t(y) < F_t(y)$ . Thus $\tilde F_t(Y_t)
\leq F_t(y)$ implies $Y_t\leq y $; i.e.
\begin{equation*}
\1_{Y_t\leq y}\geq \1_{\tilde F_t(Y_t) \leq F_t(y)}, 
\text{ uniformly in } t\in E, \, y\in \R. 
\end{equation*}
Combining the two displays, we have
\begin{equation}\label{eq:X2Y}
  \1_{Y_t\leq y}=\1_{\tilde F_t(Y_t) \leq F_t(y)},\text{ uniformly in } t\in E, \, y\in \R. 
\end{equation}

Since $\{F_t(y): t\in E, \, y\in \R\}$ is a subset of
$[0,1]$, thus if the empirical process based on
$\{w(x)(\1_{\tilde F_t(Y_t) \leq x}-x):t\in E , \, x\in
[0,1]\} $ satisfies CLT in $\ell^\infty(E\times [0,1])$,
then, by substituting $x $ with $F_t(y) $ and using
\eqref{eq:X2Y}, the empirical process based on $\{w(F_t(y))
(\1_{Y_t\leq y} - F_t(y)): t\in E , \, y\in \R \}$
satisfies the CLT in $\ell^\infty(E \times \R)$.

{Proof of $(ii)$}. Fix $t\in E$ and $y\in \R$.  If
$\tilde F_t(Y_t)\leq F_t(y)$, since $\tilde F_t(Y_t)=F_t(y)$
has probability zero, then, after throwing out this null
set, $\tilde F_t(Y_t)< F_t(y)$, which will imply $Y_t\leq
y$. If not, then $Y_t>y$, by
Lemma~\ref{lem:properites-tilde-F}, hence $F_t(y)\leq \tilde
F_t(Y_t)$. Again we have a contradiction
$F_t(y)<F_t(y)$. Thus almost surely $\1_{\tilde F_t(Y_t)
  \leq F_t(y)} \leq \1_{Y_t\leq y}$. Combining this with
\ref{eq:right} gives, almost surely,
\begin{equation}\label{eq:countable-T_0}
  \1_{Y_t\leq y}=\1_{\tilde F_t(Y_t) \leq F_t(y)},\text{ uniformly in } (t,y)\in T_0, 
\end{equation}
where $T_0$ is any countable set in $E\times \R$.
Restricting to the countable set, we have the stated
implication as in $(i)$.
\end{proof}
\begin{corollary}[cf. \cite{KKZ10}, Theorem 3]\label{cor:kkz}
  Let $Y:=\{Y_t:t\in E\}$ be a process.  Let $F_t$
  be the df of $Y_t $. In addition, assume that $F_t(\cdot)$
  is strictly increasing for each $t\in E $ and that $Y$
  satisfies the L-condition\textup{:}
\begin{equation}\label{L}
\sup_{t\in E}\Pr(\sup\limits_{s:\rho(s,t)\leq \varepsilon} |\tilde
F_t(Y_t) - \tilde F_t(Y_s)| >\varepsilon^2)\leq L
\varepsilon^2,
\end{equation} for a constant $L$ and a continuous Gaussian metric
$\rho(s,t)$ on $E$.
Then
\[ \{ \1_{Y_t\leq y} - \P(Y_t\leq y): t\in E, \, y\in \R \}
\in \CLT \text{ in } \ell^\infty(E\times \R).
\]
\end{corollary} 
\begin{remark} Under the L-condition, we will see from the
  proof of Theorem~\ref{thm:iid-pregaussian} that there is a
  countable dense set in $E\times \R$ with respect to the
  $L_2$ distance of the limiting Gaussian
  process. Hence without the restriction that $F_t(\cdot)$
  is strictly increasing, we still have a CLT but on a
  countable dense set.
\end{remark}

\begin{proof}[Proof of Corollary ~\ref{cor:kkz}]
By part ($i$) of Proposition~\ref{prop:X2Y}, we only need 
to check the conditions in Theorem~\ref{thm:main1}
with $w(x)\equiv 1$. 

Under the L-condition, we have (cf. \cite{KKZ10}, Lemma 1)
\[ \sup_x |F_t(x) - F_s(x)| \leq 2(L+1)\rho(s,t)^2.
\] Consequently by passing to the limit,
\[ \sup_x |F_t(x-) - F_s(x-)| \leq 2(L+1)\rho(s,t)^2.
\] Recalling that $\tilde F_s(x)=F_s(x-) + V(F_s(x) -
F_s(x-))$, we obtain
\begin{align*} \sup_x |\tilde F_t(x) - \tilde F_s(x)| &\leq
\sup_x |F_t(x-) - F_s(x-)|+\sup_x |V(F_t(x) - F_s(x))| \\
&\qquad +
\sup_x |V(F_t(x-) - F_s(x-))| \\ &\leq 6(L+1)\rho(s,t)^2.
\end{align*} For $t\in E$ fixed, let $A:=\{
\sup\limits_{s:\rho(s,t)\leq \varepsilon} |\tilde F_t(Y_t) -
\tilde F_t(Y_s)| >\varepsilon^2\}$. \\
On the complement, $A^c$, of $A$, we have for
all $s$ with $\rho(s,t)\leq \varepsilon$,
\begin{align*} |\tilde F_s(Y_s) - \tilde F_t(Y_t) |&\leq
|\tilde F_s(Y_s) - \tilde F_t(Y_s) | + |\tilde F_t(Y_s) -
\tilde F_t(Y_t) |\\ &\leq 6(L+1) \rho(s,t)^2 +
\varepsilon^2\\ &\leq (6L+7)\varepsilon^2.
\end{align*} Hence
\begin{align*} \Pr(\sup_{s:\rho(s,t)\leq \varepsilon}
\1_{\tilde F_s(Y_s)\leq x < \tilde F_t(Y_t)}>0)
&=\Pr(A^c,\sup_{s:\rho(s,t)\leq \varepsilon} \1_{\tilde F_s(Y_s)\leq x
< \tilde F_t(Y_t)}>0) \\
&\qquad + \Pr(A, \sup_{s:\rho(s,t)\leq
\varepsilon} \1_{\tilde F_s(Y_s)\leq x < \tilde
F_t(Y_t)}>0)\\ &\leq \P(A^c, \1_{\tilde F_t(Y_t)-( 6(L+1)
\varepsilon^2 + \varepsilon^2)\leq x < \tilde F_t(Y_t)}>0) +
L\varepsilon^2
\intertext{Keeping in mind that $\tilde F_t(Y_t)\stackrel{d}= U(0,1)$}
&\leq (7L+7)\varepsilon^2 .
\end{align*}
Similarly, $$\Pr(\sup_{s:\rho(s,t)\leq \varepsilon}
\1_{\tilde F_t(Y_t)\leq x < \tilde F_s(Y_s)}>0)\leq
(7L+7)\varepsilon^2.$$ 

In addition, obviously for $w(x) \equiv 1$
\[\lim\limits_{\alpha\to \infty} \alpha^2
\Pr(\sup_{t\in E} w(\tilde F_t(Y_t))>\alpha) = 0 .\] Thus we have
verified the conditions in Theorem~\ref{thm:main1}.
\end{proof}

We will prove Theorem~\ref{thm:main1} only for $0<x<1/2$
as explained in Remark~\ref{rem:main1}. We
will check the pre-Gaussian condition $(ii)$ and the local modulus
condition $(iii)$ in Theorem~\ref{thm:agoz}.

\subsection{Pre-Gaussian}
Let $\{G_0((s,x)): s\in E,\,x\in [0,1]\}$ be the zero mean
Gaussian process with covariance
\begin{equation}\label{eq:G_0}
 \E G_0(s,x)G_0(t,y))=w(x)w(y)\P(X_s\leq x, X_t\leq y).
\end{equation}
Under the assumptions of Theorem \ref{thm:main1}, we will
prove $G_0(s,x)$ has a version with bounded and uniformly
continuous sample paths with its $L_2$ distance $d_{G_0} $
by comparing it with some other continuous Gaussian
distance; consequently by another comparison the
centered Gaussian process with covariance
\begin{equation}\label{eq:G_c}
 \E G(s,x)G(t,y)):=w(x)w(y)[\P(X_s\leq x, X_t\leq y)-xy]
\end{equation}
has a version with bounded and uniformly continuous sample
paths with its $L_2$ distance $d_G$, which is equivalent to
say the process $\{w(y)(\1_{X_t\leq y} - y):\, t\in E, y \in
[0,1]\}$ is pre-Gaussian.
\begin{lemma}[see \cite{AGOZ}, Example 4.8]
  Let $W(y)$ be a Brownian motion and $w(y)$ as in
  Theorem~\ref{thm:main1}. Then the Gaussian process
  $\{w(y)W(y):y\in [0,1] \}$ is sample bounded and uniformly
  continuous w.r.t. its $L_2$ distance, which is given by
\begin{equation}\label{dw}
d(x,y)^2:=\E(w(y)W(y)-w(x)W(x))^2=w(x\vee y)^2|y-x|+(x\wedge
y)(w(x)-w(y))^2.
\end{equation}
\end{lemma}
\begin{lemma}\label{lem:mono} If $xw(x)^2$ is non-decreasing
and $w(x)$ is non-increasing for $0<x<\delta$, then
$$
 d(x,y)\leq d(x,z)
$$
for $0<x\leq y\leq z\leq \delta$.
\end{lemma}
\begin{proof} Let $0<x\leq y\leq z\leq \delta$.  Using
  definition \eqref{dw} and the monotonicity of $xw(x)^2$
  and $w(x)$, we obtain
\begin{align*} d(x,y)^2&=w(y)^2(y-x)+x(w(y)-w(x))^2 \\
&=xw(x)^2 + yw(y)^2 -2xw(x)w(y) \\ &\leq xw(x)^2 + zw(z)^2
-2xw(x)w(z) \\ &=d(x,z)^2. \qedhere
\end{align*}
\end{proof}

Next we give an upper bound for $d_{G_0}$ under WL-condition
in Theorem~\ref{thm:main1}.

\begin{lemma}\label{lem:upper} Let $d(x,y)$ be as in
  \eqref{dw} and $d_{G_0}((s,x),(t,y))$ the $L_2$ distance
  of the Gaussian process $G_0$ in \eqref{eq:G_0}.  Then under
  the WL-condition, we have
$$
 d_{G_0}^2((s,x), (t,y))\leq 2d^2(x,y)+4L\rho(s,t)^2.
$$
\end{lemma}
\begin{proof} First observe that for $t\in E$
\begin{equation}\label{eq:d}
d(x,y)^2=\E(w(y)W(y)-w(x)W(x))^2=\E |w(x)\1_{X_t \leq x} -
w(y)\1_{X_t \leq y}|^2.
\end{equation} Using, by the WL-condition for fixed $s$ and $t$,
\begin{equation}\label{eq:wl} \P(X_s \leq x < X_t)\leq
\frac{L\rho(s,t)^2}{w(x)^2} \text{ and } \P(X_t \leq x <
X_s)\leq \frac{L\rho(s,t)^2}{w(x)^2},
\end{equation} 
we obtain
\begin{align} &\quad d_{G_0}((s,x),(t,y))^2 \label{eq:d_G_0}\\ 
&\quad= \E
|w(x)\1_{X_s \leq x} - w(y)\1_{X_t \leq y}|^2 \nonumber \\ 
&\quad= \E
|w(x)\1_{X_s \leq x} - w(x)\1_{X_t \leq x} + w(x)\1_{X_t
\leq x} - w(y)\1_{X_t \leq y}|^2 \nonumber\\ 
&\quad\leq 2\E
|w(x)\1_{X_s \leq x} - w(x)\1_{X_t \leq x}|^2 + 2\E
|w(x)\1_{X_t \leq x} - w(y)\1_{X_t \leq y}|^2 \nonumber\\ 
&\quad=
2w(x)^2\E |1_{X_s \leq x} - \1_{X_t \leq x}|^2 + 2d(x,y)^2
\qquad \text{ by } \eqref{eq:d} \nonumber\\ 
&\quad \leq2w(x)^2(\P(X_s
\leq x < X_t)+ \P(X_t \leq x < X_s)) + 2d(x,y)^2\nonumber \\ 
&\quad
\leq 4L\rho(s,t)^2 + 2d(x,y)^2 \qquad \text{ by }
\eqref{eq:wl}. \nonumber \qedhere
\end{align}
\end{proof}

\begin{corollary}\label{cor:preg} Under the WL-condition,
  the process $G_0(t,y)$ is sample bounded and uniformly
  continuous with respect to its $L_2$ distance; the same is true for
a zero mean Gaussian process with covariance 
\begin{equation}\label{eq:G}
 \E G(s,x)G(t,y)):=w(x)w(y)[\P(X_s\leq x, X_t\leq y)-xy].
\end{equation}
\end{corollary}
\begin{proof}
  By assumption, $\rho$ is the $L_2$ distance of a zero mean
  Gaussian process on $E$, say $\{H_0(t):t\in E\} $, with
  bounded and uniformly $\rho$-continuous sample paths.  Let
  the metric $d$ on $ [0,1]$ as given in~\ref{dw} with the
  corresponding Gaussian process $w(x)W(x)$, which is sample
  bounded and uniformly $d$-continuous.  Let
  $H_2((t,y)):=2^{1/2}w(y)W(y) + 2L^{1/2}H_0(t): t\in E, \,
  y\in [0,1]$, where $W$ and $H_0$ are independent.  Then
  the $L_2$ distance, $d_{H_2}((s,x),(t,y))$, of $H_2$ is
  $2^{1/2}d(x,y) + 2L^{1/2}\rho(s,t)$.  Total boundedness of
  $d$ and $\rho$ implies that of $d_{H_2}$.  Thus let $T_0$
  be a dense subset in $(E\times [0,1], d_{H_2})$; since
  $d_{G_0}\leq d_{H_2}$ by \eqref{eq:d_G_0}, $T_0$ is also a
  dense subset in $(E\times [0,1], d_{G_0})$; Using the
  comparison Theorem~\ref{thm:comp} with $H_1:=G_0$ and that
  $d_{G_0}\leq d_{H_2}$, the Gaussian process $\{G_0:
  (s,x)\in T_0\}$ is sample bounded and uniformly $d_{G_0}$
  continuous. By Lemma~\ref{lem:extension}, $\{G_0: (s,x)\in
  E\times R\}$ is sample bounded and uniformly
  $d_{G_0}$-continuous; the second statement in the Lemma is
  straightforward.
\end{proof}


For the pre-Gaussian property of the empirical process
considered in \cite{KKZ10}, we give a different proof rather
than the constructive one in \cite{KKZ10} using the generic chaining
\cite{TGC}.
\begin{theorem}
  \label{thm:iid-pregaussian}
  Let $\{Y(t):t\in E\}$ be a process and satisfies the L-condition,
  then the centered Gaussian process on $E\times \R$ with
  covariance either
$$\P(Y_s\leq x, Y_t\leq y) - \P(Y_s\leq x)\P(Y_t\leq y) $$ or
$$\P(Y_s\leq x, Y_t\leq y) $$
has a version, which is sample bounded and uniformly
continuous with respect to its $L_2$ distance.
\end{theorem}
\begin{proof}
  Let $\{G_1(t,y):t\in E, \, y\in \R\}$ and $\{G_2(t,y):t\in
  E, \, y\in \R\}$ be the Gaussian processes on $E\times \R$
  with covariance $\P(Y_s\leq x, Y_t\leq y) - \P(Y_s\leq
  x)\P(Y_t\leq y) $ and $\P(Y_s\leq x, Y_t\leq y)$,
  respectively.  Let $d_{G_1}$ and $d_{G_2}$ be their $L_2$
  distances, respectively; i.e,
 \begin{align}
   d_{G_1}((s,x),(t,y))^2 &=\E (\1_{Y_s\leq x}-\1_{Y_t\leq
     y})^2
   - (\E (\1_{Y_s\leq x}-\1_{Y_t\leq y}))^2, \label{eq:d_1-d_2}\\
   d_{G_2}((s,x),(t,y))^2 &=\E (\1_{Y_s\leq x}-\1_{Y_t\leq
     y})^2.  \nonumber
 \end{align}
 And,
  \begin{align}
    d_{G_2}((s,x),(t,y))^2
    &=\E (\1_{Y_s\leq x}-\1_{Y_t\leq y})^2 \label{eq:d_1<d_H}\\
    &=\E (\1_{Y_s\leq x}-\1_{Y_t\leq x}+\1_{Y_t\leq x}-\1_{Y_t\leq y})^2 \nonumber\\
    &\leq 2\E (\1_{Y_s\leq x}-\1_{Y_t\leq x})^2+\E(\1_{Y_t\leq x}-\1_{Y_t\leq y})^2 \nonumber\\
    &\leq 2 (\P(Y_s\leq x<Y_t)+\P(Y_t\leq x<Y_s))+|F_t(y)-F_t(x)| \nonumber\\
    &\leq 6(L+1)\rho(s,t)^2+|F_t(y)-F_s(x)|, \nonumber
 \end{align}
 where in the last line of the above display, we used
 Lemma~1 in \cite{KKZ10}.
 
 Let $W(\cdot)$ be a Brownian motion on $[0,\infty)$.
 Define the centered Gaussian process $$H_2(t,y):=W(F_t(y)):
 t\in E, \, y\in \R,$$ where $F_t(\cdot)$ be the df of
 $Y_t$.  Then its $L_2$ distance
 $d_{H_2}((s,x),(t,y))=|F_t(y)-F_s(x)|^{1/2}$.
 By the uniform continuity of the sample
 paths of $W(\cdot)$ on $[0,1]$,
 it follows
 that $H_2$ is sample bounded and uniformly continuous with
 respect to $d_{H_2}$.  By the L-condition, let
 $\{H_1(t):t\in E\}$, independent from $H_2$, be a Gaussian
 process with bounded and uniformly continuous sample paths
 with it's $L_2$ distance $\rho$. Define $H(t,y)=H_2(t,y) +
 (6L+6)^{1/2}H_1(t)$.  Then $\{H(t,y): t\in E, \, y\in \R\}$
 is sample bounded and uniformly continuous with respect to
 it's $L_2$ distance $d_H$.
 Total boundedness of $d_{H_1}$ and $d_{H_2}$ implies that
 of $d_H$ as can be seen from the equation
$$d_H((t_1, y_1), (t_2, y_2))^2=d_{H_2}((t_1, y_1), (t_2, y_2))^2 + (6L+6)d_{H_1}(t_1,t_2)^2.$$ 
Thus let $T_0 $ be a countable dense subset in $(E\times \R,
d_H) $.  Since $d_{G_1}\leq d_H $ in view of
\eqref{eq:d_1-d_2} and \eqref{eq:d_1<d_H}, by the comparison
theorem~\ref{thm:comp}, $\{G_1(s,x): (s,x)\in T_0\}$ is
sample bounded and uniformly continuous with respect to
$d_{G_1}$.  Since $T_0$ is also dense in $(E\times \R,
d_{G_1}) $, by Lemma~\ref{lem:extension}, $\{G_1(s,x):
(s,x)\in E\times \R\}$ has a version which is sample
bounded and uniformly $d_{G_1}$-continuous.
\end{proof}

\subsection{Local modulus} Recall that a positive function
$L(x) $ defined on $(0,\infty) $ is slowly varying at
infinity (in a neighborhood of zero) if $ {L(\lambda x)}/{L(x)}\to 1, x\to
\infty \;(x \to 0) $ for every $\lambda>0$ (see
\cite[p. 276]{FellerII}). One says a function $U(x)$ is
regularly varying at
infinity (in a neighborhood of zero) if $U(x)=x^\rho L(x) $ for some
$-\infty<\rho<\infty $, and some slowly varying at
infinity (in a neighborhood of zero) function
$L(x)$; $\rho$ is called the exponent (see
\cite[p. 275]{FellerII}).
\begin{lemma}\label{lem6} Let $w(x)>0$ for $0<x\leq 1/2$ and is
  regularly varying in a neighborhood of $0$ with nonzero
  exponent $\alpha$. Let $\theta_0>0$ be small enough such that
  $w(x)$ is non-increasing for $0<x<\theta_0$.  Then for
  $0<\theta<\theta_0$
\[ \sum_{k=0}^\infty \frac{1}{w(2^{-k}\theta)^2} \leq
\frac{C}{w(\theta)^2},
\] where $C$ depends only on the weight function $w(x)$, but
not on the argument $x$.
\end{lemma}
\begin{proof}

Since $w(x)$ is non-increasing for $0<x<\theta_0$,
\[ (\ln 2)\sum_{k=1}^\infty \frac{1}{w(2^{-k}\theta)^2} \leq
\int_0^\theta \frac{1}{w(y)^2}\frac{dy}{y} \leq (\ln
2)\sum_{k=0}^\infty \frac{1}{w(2^{-k}\theta)^2}.
\] By Theorem 1 in \cite[p. 281]{FellerII}, we have
\[ \frac{\tfrac{1}{w(\theta)^2}} {\int_0^\theta
\frac{1}{w(y)^2}\frac{dy}{y}} \to \alpha, \quad \hbox{ as
}\theta\rightarrow 0,
\] where $\alpha > 0 $ is the exponent of the regularly
varying function $1/w(x)^2$ (note that if $w(x)$ is
regularly varying, so is $1/w(x)^2$). Therefore, there is a
constant $C(w)$ such that
\[ \Big|\frac {\int_0^\theta
\frac{1}{w(y)^2}\frac{dy}{y}}{\tfrac{1}{w(\theta)^2}} \Big|
\leq C(w), \quad 0<\theta<\theta_0. \qedhere
\]
\end{proof}

\begin{lemma}\label{lem:ab} Given $\varepsilon>0$, under the
  assumptions of Theorem~\ref{thm:main1}, we have for
  $0<a<b<1$ and $t$ fixed
\[ \Pr(\exists s, \rho(s,t)\leq \varepsilon, \exists x\in
(a,b] : X_s\leq x <X_t)\leq
\frac{C\varepsilon^2}{w(b)^2}+(b-a),
\] and
\[ \Pr(\exists s, \rho(s,t)\leq \varepsilon, \exists x\in
(a,b] : X_t\leq x <X_s)\leq
\frac{C\varepsilon^2}{w(b)^2}+(b-a),
\] where $C$ is a constant depending only on the function
$w(x)$.
\end{lemma}
\begin{proof} Let $N\geq 0$ be the biggest integer such that
  $ b/2^N \geq a$. Then,
\begin{align*}
& \Pr(\exists s, \rho(s,t)\leq
\varepsilon,\exists x\in (a,b]: X_s\leq x <X_t)\\
&\leq \sum_{k=0}^{N-1} \Pr(\exists s, \rho(s,t)\leq \varepsilon,
\exists x\in (2^{-k-1}b, 2^{-k}b]: X_s\leq x <X_t) \\
&\qquad + \Pr(\exists s, \rho(s,t)\leq \varepsilon: X_s\leq x <X_t)\\
&\leq \sum_{k=0}^{N-1}
\Pr(\exists s, \rho(s,t)\leq \varepsilon: X_s\leq 2^{-k}b <X_t)+
\sum_{k=0}^{N-1} \P(2^{-k-1}b<X_t\leq 2^{-k}b)\\
&\qquad + \Pr(\exists s, \rho(s,t)\leq \varepsilon: X_s\leq 2^{-N}b <X_t) + \P(a<X_t\leq 2^{-N}b)\\
&\leq \sum_{k=0}^{N-1} \Pr(\exists s, \rho(s,t)\leq
\varepsilon: X_s\leq
2^{-k}b <X_t)+ \sum_{k=0}^{N-1}(2^{-k}b - 2^{-k-1}b)\\
&\qquad + \Pr(\exists s, \rho(s,t)\leq \varepsilon: X_s\leq 2^{-N}b <X_t) + 2^{-N}b-a \\
&\leq \sum_{k=0}^{N} \Pr(\exists s, \rho(s,t)\leq \varepsilon
: X_s\leq 2^{-k}b <X_t) + \sum_{k=0}^{N-1}(2^{-k}b -
2^{-k-1}b) + 2^{-N}b-a \\ &\leq \sum_{k=0}^\infty
\frac{L\varepsilon^2}{w( 2^{-k}b)^2} + (b-a) \quad\text{
using {WL-condition} to bound the probabilities}\\ &\leq
\frac{C\varepsilon^2}{w(b)^2} +(b-a) \quad\text{ by
Lemma~\ref{lem6}. }
\end{align*} 
The proof for the second part is similar; just
change from $X_t\leq x <X_s $ for $ 2^{-k-1}b <x\leq 2^{-k}b
$ to $X_t\leq 2^{-k-1}b <X_s $, with the same exceptional
probability $(2^{-k}b-2^{-k-1}b)$.
\end{proof} 

For the following, we use $C$ to denote a constant which may
change from line to line and depends only on the weight
function $w(x)$.

Let the distance $d$ be as in \eqref{dw}. Then,
$$e((s,x),(t,y)):= \max\{d(x,y), \rho(s,t)\}$$
is bounded by the Gaussian distance $ (d(x,y)^2+
\rho(s,t)^2)^{1/2}$ on $E\times (0,1)$ and will be used
as the \lq $\rho$\rq \,in (iii) of Theorem \ref{thm:agoz}.
\begin{lemma}\label{lem:x0} For $t\in E$, $y\in (0,1)$,
let $x_0:=\inf\{x: \text{ for some } s,
e((s,x),(t,y))<\varepsilon\}$, then
\begin{equation}\label{eq:x0} d(x_0,y)\leq \varepsilon.
\end{equation}
\end{lemma}
\begin{proof} Indeed there exist a sequence $(s_n,
  x_n)_{n\in \N}$ in the set over which the infimum is taken
  such that $|x_n -x_0|\rightarrow 0$ as $n\rightarrow
  \infty$ and that $d(x_n,y)\leq \varepsilon$.  By the
  sample continuity of the weighted Wiener process
  $w(x)W(x)$, we have $d(x_n, y)\rightarrow d(x_0,
  y)$ as $n\rightarrow \infty$. Hence we have obtained
  $d(x_0,y)\leq \varepsilon$.
\end{proof}
\begin{remark*} The finiteness of $d(x_0,y)$ implies that
  $x_0$ can't be zero in view of \eqref{dw} since $w(x)\to
  \infty$ and $xw(x)^2\to 0$ as $x\to 0$.
\end{remark*}
\begin{lemma}\label{lem:x1} For $t\in E$, $y\in (0,1)$,
  let $x_1:=\sup\{x: \text{ for some } s,
  e((s,x),(t,y))<\varepsilon\}$, then
\begin{equation}\label{eq:x1} d(y,x_1)\leq \varepsilon.
\end{equation}
\end{lemma}
\begin{proof} By a similar argument as in the proof of the
  previous lemma.
\end{proof}

The following Lemma \ref{lem:sty}, Lemma \ref{lem:gty} and Lemma \ref{lem:iii}
constitute a weighted version of
Lemma 4 in \cite{KKZ10}.
For brevity of notation, for fixed $(t,y)\in E\times [0,1]$,
we write $(s,x):e<\varepsilon, x\leq y $ for the set
$\{(s,x):e((s,x),(t,y))<\varepsilon, x\leq y\} $.
\begin{lemma} \label{lem:sty} Under the assumptions of
  Theorem ~\ref{thm:main1}, we have for all $\varepsilon>0 $
  and $(t,y)\in E\times [0,1]$,
\[ w(y)^2\Pr(\sup_{(s,x):e<\varepsilon, x\leq y} |\1_{X_s\leq
x}-\1_{X_t\leq x}|>0)\leq C\varepsilon^2.
\]
\end{lemma}
\begin{proof} Let $x_0$ be as in Lemma \ref{lem:x0}. Then,
\begin{align*} &w(y)^2\Pr(\sup_{(s,x):e<\varepsilon, x\leq y}
|\1_{X_s\leq x}-\1_{X_t\leq x}|>0)\\ &=w(y)^2\Big(\Pr(\exists
(s,x),e((s,x),(t,y))<\varepsilon, x\leq y: X_s\leq x <X_t)
\\ & \qquad \qquad + \Pr(\exists
(s,x),e((s,x),(t,y))<\varepsilon, x\leq y: X_t\leq x
<X_s)\Big)\\ & =w(y)^2\Big(\Pr(\exists
(s,x),e((s,x),(t,y))<\varepsilon, x\in (x_0,y]: X_s\leq x
<X_t ) \\ & \qquad \qquad + \Pr(\exists
(s,x),e((s,x),(t,y))<\varepsilon, x\in (x_0,y]: X_t\leq x
<X_s)\Big)\\ &\leq w(y)^2(C\varepsilon^2/w(y)^2+(y-x_0))
\mbox{ by Lemma ~\ref{lem:ab} } \\ &\leq C\varepsilon^2.
\end{align*} For the last inequality, we used
\[ w(y)^2(y-x_0)\leq d(x_0,y)^2\leq \varepsilon^2 \quad
\text{ by } \eqref{eq:x0}. \qedhere
\]
\end{proof}
\begin{lemma} \label{lem:gty} Under the assumptions of
  Theorem ~\ref{thm:main1}, we have for all $\varepsilon>0 $
  and $(t,y)\in E\times [0,1]$,
\[ w(x_1)^2\Pr(\sup_{(s,x):e<\varepsilon, x> y} |\1_{X_s\leq
x}-\1_{X_t\leq x}|>0)\leq C\varepsilon^2.
\]
\end{lemma}
\begin{proof} Let $x_1$ be as in Lemma \ref{lem:x1}. Then,
  \begin{align*} &w(x_1)^2\Pr(\sup_{(s,x):e<\varepsilon, x>
      y}
    |\1_{X_s\leq x}-\1_{X_t\leq x}|>0)\\
    &=w(x_1)^2\Big(\Pr(\exists
    (s,x),e((s,x),(t,y))<\varepsilon, x> y: X_s\leq x <X_t)
    \\ & \qquad \qquad + \Pr(\exists
    (s,x),e((s,x),(t,y))<\varepsilon, x> y: X_t\leq x
    <X_s)\Big)\\ & =w(x_1)^2\Big(\Pr(\exists
    (s,x),e((s,x),(t,y))<\varepsilon, x\in (y, x_1]: X_s\leq
    x <X_t ) \\ & \qquad \qquad + \Pr(\exists
    (s,x),e((s,x),(t,y))<\varepsilon, x\in (y, x_1]: X_t\leq
    x <X_s)\Big)\\ &\leq
    w(x_1)^2(C\varepsilon^2/w(x_1)^2+(x_1 - y)) \mbox{ by
      Lemma ~\ref{lem:ab} } \\ &\leq C\varepsilon^2.
  \end{align*} 
For the last inequality, we used
\[ w(x_1)^2(x_1 - y)\leq d(y, x_1)^2\leq \varepsilon^2 \quad
\text{ by } \eqref{eq:x1}. \qedhere
\]
\end{proof}
In the following lemma, for fixed $(t,y)\in E\times [0,1]$,
we write $\sup\limits_{(s,x):e<\varepsilon}$ for
$\sup\limits_{\{(s,x):e((s,x), (t,y))<\varepsilon\}}$ and
the same applies to other similar quantities.
\begin{lemma}\label{lem:iii} Under the assumptions of
Theorem ~\ref{thm:main1}, we have for all $\varepsilon>0 $
and $(t,y)\in E\times [0,1]$,
\[ \sup_{\alpha>0} \alpha^2\Pr(\sup_{(s,x):e<\varepsilon}
|w(x)\1_{X_s\leq x}-w(y)\1_{X_t\leq y}|>\alpha)\leq
C\varepsilon^2.
\]
\end{lemma}
\begin{proof} We split the quantity:
$$
 w(x)\1_{X_s\leq x} - w(y)\1_{X_t\leq y} =[w(x)\1_{X_t\leq
x}-w(y)\1_{X_t\leq y}] + [w(x)(\1_{X_s\leq x} - \1_{X_t\leq
x})].
$$
Consider the weak $L_2$ norms of the components:
\begin{gather} A:=\sup_{\alpha>0}
\alpha^2\Pr(\sup_{(s,x):e<\varepsilon} |w(x)\1_{X_t\leq
x}-w(y)\1_{X_t\leq y}|>\alpha) \label{eq:A} \\
B:=\sup_{\alpha>0} \alpha^2\Pr(\sup_{(s,x):e<\varepsilon}
w(x)|\1_{X_s\leq x} - \1_{X_t\leq x}|> \alpha). \label{eq:B}
\end{gather} First we estimate A. Since
\begin{multline*} \sup_{\alpha>0}
\alpha^2\Pr(\sup_{(s,x):e<\varepsilon} |w(x)\1_{X_t\leq
x}-w(y)\1_{X_t\leq y}| >\alpha)\\ \leq \sup_{\alpha>0}
\alpha^2\Pr(\sup_{x:d(x,y)<\varepsilon} |w(x)\1_{X_t\leq
x}-w(y)\1_{X_t\leq y}| >\alpha)
\end{multline*} and $t$ is fixed, this is the case in
Example 4.9 in~\cite{AGOZ}. Hence we have
\begin{equation}\label{A} A:=\sup_{\alpha>0}
\alpha^2\Pr(\sup_{(s,x):e<\varepsilon} |w(x)\1_{X_t\leq
x}-w(y)\1_{X_t\leq y}|>\alpha)\leq C\varepsilon^2.
\end{equation} Now we consider B.  Since
\begin{multline*} \sup_{\alpha>0}
\alpha^2\Pr(\sup_{(s,x):e<\varepsilon} w(x)|\1_{X_s\leq x} -
\1_{X_t\leq x}|> \alpha)\leq \sup_{\alpha>0}
\alpha^2\Pr(\sup_{(s,x):e<\varepsilon, x\leq y}
w(x)|\1_{X_s\leq x} - \1_{X_t\leq x}|> \alpha) \\ \qquad +
\sup_{\alpha>0} \alpha^2\Pr(\sup_{(s,x):e<\varepsilon, x>y}
w(x)|\1_{X_s\leq x} - \1_{X_t\leq x}|> \alpha),
\end{multline*} it suffices to consider bounds of the
last two quantities.  Without loss of generality, we assume
$w(x)$ is monotone on $(0,1/2]$. For $\alpha>0$, let
$$x_\alpha =\sup\{x\in
[0,1/2]: w(x) > \alpha\}.$$ 

\medskip
{\noindent  \emph{Case $x\leq y$}. }\\
Recall $x_0=\inf\{x:e((s,x),(t,y))<\varepsilon\} $.  First
we consider the extreme cases for $x_\alpha$.
\item{(1).} By continuity of $w(\cdot)$,
  if $x_\alpha >y$, then $\alpha\leq w(y)$,
consequently
\begin{multline*} \sup_{\alpha<w(y)}
\alpha^2\Pr(\sup_{(s,x):e<\varepsilon, x\leq y}
w(x)|\1_{X_s\leq x} - \1_{X_t\leq x}|> \alpha)\\ \leq
w(y)^2\Pr(\sup_{(s,x):e<\varepsilon, x\leq y} |\1_{X_s\leq
x}-\1_{X_t\leq x}|>0) \leq C\varepsilon^2 \text{  by Lemma~\ref{lem:sty}}.
\end{multline*}

\item{(2).}  If $x_\alpha\leq x_0$, then $w(x_0)\leq
\alpha$, hence $w(x)\leq \alpha$ for $x_0\leq x $.  For
$\alpha$ such that $x_\alpha\leq x_0$, the event under the probability of \eqref{eq:B} is
empty.

\item{(3).} Now $x_0< x_{\alpha} \leq y$. In this case,
$w(y)\leq \alpha<w(x_0)$.  Take $\varepsilon>0$.  We have
\begin{align*} &B:=\sup_{w(y)\leq \alpha<w(x_0)}
\alpha^2\Pr(\sup_{(s,x):e<\varepsilon} w(x)|\1_{X_s\leq x} -
\1_{X_t\leq x}|> \alpha)\\
&\quad \leq \sup_{w(y)\leq
\alpha<w(x_0)} \alpha^2\Pr(\sup_{(s,x):e<\varepsilon}
w(x)\1_{X_s\leq x< X_t}> \alpha)\\
&\qquad +\sup_{w(y)\leq
\alpha<w(x_0)} \alpha^2\Pr(\sup_{(s,x):e<\varepsilon}
w(x)\1_{X_t\leq x < X_s}> \alpha)\\
&\quad =I+II.
\end{align*} For $I$,
\begin{align*} &I=\sup_{w(y)\leq
\alpha<w(x_0)}\alpha^2\Pr(\sup_{(s,x):e<\varepsilon}
w(x)\1_{X_s \leq x<X_t} > \alpha) \\ &\quad \leq \sup_{x_0<
x_{\alpha}\leq y}w(x_\alpha)^2\Pr(\sup_{(s,x):e<\varepsilon}
w(x)\1_{X_s \leq x< X_t}> \alpha) \\ &\quad \leq \sup_{x_0<
x_{\alpha}\leq y}w(x_\alpha)^2\Pr(\sup_{(s,x):e<\varepsilon}
\1_{ X_s \leq x< X_t, x\leq x_\alpha}>0) \\ &\quad \leq
\sup_{x_0< x_{\alpha}\leq y}w(x_\alpha)^2 \big (
C\varepsilon^2/w(x_\alpha)^2 + (x_\alpha - x_0)\big ) \mbox{
using Lemma~\ref{lem:ab} } \\ &\quad \leq C\varepsilon^2.
\end{align*} For the last inequality, we used
\[ w(x_\alpha)^2 (x_\alpha - x_0) \leq d(x_0,
x_\alpha)^2\leq d(x_0,y)^2\leq \varepsilon^2
\] of Lemma~\ref{lem:mono} and Lemma~\ref{lem:x0}.

$II$ can be handled in the same way.

{\noindent  \emph{Case $x> y$}. }\\
Recall
$x_1=\sup\{x:e((s,x),(t,y))<\varepsilon\} $.  First we
consider the extreme cases for $x_\alpha$.
\item{(1).}  By continuity of $w(\cdot)$,
  if $x_\alpha >x_1$, then $\alpha\leq w(x_1)$,
consequently
\begin{multline*} \sup_{\alpha<w(x_1)}
\alpha^2\Pr(\sup_{(s,x):e<\varepsilon, x> y} w(x)|\1_{X_s\leq
x} - \1_{X_t\leq x}|> \alpha)\\ \leq
w(x_1)^2\Pr(\sup_{(s,x):e<\varepsilon, x> y} |\1_{X_s\leq
x}-\1_{X_t\leq x}|>0) \leq C\varepsilon^2.
\end{multline*} by Lemma~\ref{lem:gty}. 
consider $\alpha\geq w(y)$, i.e. $x_\alpha\leq y$.
\item{(2).}  If $x_\alpha\leq y$, then $w(y)\leq \alpha$,
hence $w(x)\leq \alpha$ for $y\leq x $.  For $\alpha$ such that
$x_\alpha\leq y$, 
the event under the probability of \eqref{eq:B} is empty.

\item{(3).} Now $y< x_{\alpha} \leq x_1$. In this case,
$w(x_1)\leq \alpha<w(y)$.  Take $\varepsilon>0$.  We have
\begin{align*} &B:=\sup_{w(x_1)\leq \alpha<w(y)}
\alpha^2\Pr(\sup_{(s,x):e<\varepsilon} w(x)|\1_{X_s\leq x} -
\1_{X_t\leq x}|> \alpha)\\
&\quad \leq \sup_{w(x_1)\leq
\alpha<w(y)} \alpha^2\Pr(\sup_{(s,x):e<\varepsilon}
w(x)\1_{X_s\leq x< X_t}> \alpha)\\
&\qquad
+\sup_{w(x_1)\leq \alpha<w(y)}
\alpha^2\Pr(\sup_{(s,x):e<\varepsilon} w(x)\1_{X_t\leq x <
X_s}> \alpha)\\
&\quad =I+II.
\end{align*} For $I$,
\begin{align*} &I=\sup_{w(x_1)\leq
\alpha<w(y)}\alpha^2\Pr(\sup_{(s,x):e<\varepsilon}
w(x)\1_{X_s \leq x<X_t} > \alpha) \\ &\quad \leq \sup_{y<
x_{\alpha}\leq
x_1}w(x_\alpha)^2\Pr(\sup_{(s,x):e<\varepsilon} w(x)\1_{X_s
\leq x< X_t}> \alpha) \\ &\quad \leq \sup_{y< x_{\alpha}\leq
x_1}w(x_\alpha)^2\Pr(\sup_{(s,x):e<\varepsilon} \1_{ X_s \leq
x< X_t, x\leq x_\alpha}>0) \\ &\quad \leq \sup_{y<
x_{\alpha}\leq x_1}w(x_\alpha)^2 \big (
C\varepsilon^2/w(x_\alpha)^2 + (x_\alpha - y)\big ) \mbox{
using Lemma~\ref{lem:ab} } \\ &\quad \leq C\varepsilon^2.
\end{align*} For the last inequality, we used
\[ w(x_\alpha)^2 (x_\alpha - y) \leq d(y, x_\alpha)^2\leq
d(y, x_1)^2\leq \varepsilon^2
\] of Lemma~\ref{lem:mono} and Lemma~\ref{lem:x1}.

$II$ can be handled in the same way.  Hence we have $B\leq
C\varepsilon^2 $. This together with \eqref{A} completes the
proof.
\end{proof}
\begin{proof}[Proof of Theorem~\ref{thm:main1}] We apply
  Theorem~\ref{thm:agoz} to the process
$\{w(y)(\1_{X_t\leq y} - y):\, t\in E, y\in [0,1]\} $.

Since for each $s\in E$, $X(s)$ takes values on $(0,1)$
and $xw(x)\rightarrow 0$ as $x\to 0$, almost surely
$$\sup\limits_{s\in E, \, x\in [0,1/2]}w(x)|\1_{X_s\leq x}-x|<\infty.$$ 
Also we observe for each $s\in E, x\in [0,1/2]$
$$\P(w(x)(\1_{X_s\leq x}-x))^2<\infty.$$

Since $w(x)$ is decreasing near $0$,
\begin{align*} \lim_{\alpha\to \infty}\alpha^2\Pr(\sup_{s\in
E, x\in (0,1/2]}w(x)\1_{X_s\leq x} >\alpha) &\leq
\lim_{\alpha\to \infty}\alpha^2\Pr(\sup_{s\in E}w(X_s)
>\alpha)\\ &=0 \text{ by assumption of
Theorem~\ref{thm:main1}},
\end{align*} 
which in turn implies
\begin{align*} \lim_{\alpha\to \infty}\alpha^2\Pr(\sup_{s\in
E, x\in (0,1/2]}w(x)|\1_{X_s\leq x}-x|>\alpha) =0.
\end{align*} 
This verifies $(i)$ in Theorem
\ref{thm:agoz}. Corollary~\ref{cor:preg} verifies the pre-Gaussian 
condition $(ii)$. 

In view of Lemma~\ref{lem:iii} and the inequality
\[
 \Lambda_{2,\infty}(f+g)\leq C  (\Lambda_{2,\infty}(f) +
 \Lambda_{2,\infty}(g))
\]
where $\Lambda_{2,\infty}(f):=[\sup_{t>0} t^2\P(\{|f|>t\})]^{1/2}$
for some constant $C$, to verify the local modulus condition
($iii$) in Theorem \ref{thm:agoz} for the functions $w(x)(\1_{X_s\leq x}-x)$,
it is enough to have
\begin{equation}\label{eq:wl2}
 \sup_{\alpha>0} \alpha^2 \Pr( \sup_{d(x,y)\leq \epsilon} |w(x)x -w(y)y|>\alpha) \leq K \epsilon^2
\end{equation}
for some constant K.
W.o.l.g, assume $x<y$. 
Inequality \ref{eq:wl2} follows from 
\begin{align*}
  |xw(x) - yw(y)|^2 &\leq 2x^2(w(x) - w(y))^2 + 2 w(y)^2(y-x)^2 \\
                    &\leq 2x(w(x) - w(y))^2 + 2 w(y)^2(y-x) \\
                    &=2d(x,y)^2 \text{ by \ref{dw} }\\
                    &\leq 2\epsilon^2.\qedhere
\end{align*}
\end{proof}

\section{An example} \label{s:example} A special class of
uniform processes (copula processes) can be obtained from
distributional transforms. Specifically, given a process
$Y:=\{Y_t:t\in E\} $, define $X:=X_t:=\tilde F_t(Y_t) $,
where $\tilde F_t(\cdot)$ is the distributional transform of
the df of $Y_t$.  Now, we give an example as an application of
Theorem~\ref{thm:main1} when $\{Y_t:t\in E\}=\{B_t:t\in
[1,2]\} $, where $B_t$ is a Brownian motion.
\begin{theorem}\label{thm:ex} Let $\{B_t:t\geq 0\}$ be a
  Brownian motion and $F_t(x)$ be the distribution function
  of $B_t$.  Let $w(x)=x^{-\alpha}L(x)$, for $0<x<1/2$,
  $0<\alpha<1/2$, and $L(x)$ slowly varying at $0$ and
  assume $w(x)$ is symmetric about $1/2 $. Further assume
  that $w(x)$ is non-increasing and
  $xw(x)^2$ non-decreasing near $0$. Then
\[ \{ w(F_t(y))(\1_{B_t\leq y} - F_t(y)):t\in [1,2], \, y\in \R
\} \in \CLT \text{ in } \ell^\infty([1,2]\times \R).
\]
\end{theorem}
\begin{remarks} The interval $[1,2]$ can be replaced by
  any interval $[a,b]$ provided $a>0$, which can be seen from the proof of
  the above theorem; also a priori, we need $F_t(\cdot)$ be strictly increasing.
\end{remarks}
We will verify the conditions in Proposition
\ref{prop:X2Y} to prove this theorem at the end of this
section.  To this end, we start with some lemmas. For the
following, let $\phi(x)= (2\pi)^{-1/2}e^{-x^2/2}$ and
$\Phi(y):= (2\pi)^{-1/2} \int_{-\infty}^y e^{-s^2/2}\,ds$.

\begin{lemma}[\cite{Feller68}, p. 175]\label{lem:tail} For
$y>0$,
\[ y^{-1}(1-y^{-2})(2\pi)^{-1/2}e^{-y^2/2}\leq \Phi(-y) \leq
y^{-1}(2\pi)^{-1/2}e^{-y^2/2}.
\] In particular, for $y>\sqrt 2$,
\[ 2^{-1}y^{-1}(2\pi)^{-1/2}e^{-y^2/2}\leq \Phi(-y) \leq
y^{-1}(2\pi)^{-1/2}e^{-y^2/2}.
\]
\end{lemma}
\begin{lemma}[\cite{Seneta76}, p. 18] \label{lem:sen} Let
$L(x)$ be a slowly varying function at $0$, then for any
$\gamma>0$,
$$x^\gamma L(x)\rightarrow 0, x^{-\gamma} L(x)\rightarrow \infty \text{ as }x\rightarrow 0. $$
Consequently, for $0<\gamma_1<2\alpha<\gamma_2<1 $ and a
function $L(x)$ slowly varying {\normalfont (at $0$)}, there
are constants $c_1$, $c_2$,
\[ c_1x^{\gamma_2}\leq x^{2\alpha}/L(x) \leq
c_2x^{\gamma_1}, \quad 0<x<1/2.
\]
\end{lemma}

For $c>0$, let $L_c(x)=\exp(c\sqrt{\ln(1/x)})$.
\begin{lemma}\label{lem:L} The function $L_c(x)$ is slowly
varying at $0$; that is for all $\lambda>0$
\[ \lim_{x\rightarrow 0}\frac{L_c(\lambda x)}{L_c(x)} =1.
\]
\end{lemma}
\begin{proof} By definition.
\end{proof}
\begin{lemma}\label{lem:y} For $0<x<1/4$, let
$y=-\Phi^{-1}(x)$. Then
$$y\leq \sqrt{2\ln(1/x)}$$ and
$$ \phi(-\Phi^{-1}(x) + c) \leq  Cx L_C(x) \qquad\quad \text{ for } c<0, $$
$$\phi(-\Phi^{-1}(x) + c) \leq 2^{3/2}x\sqrt{\ln(1/x)} \quad \text{ for } c\geq 0,$$
where $C$ depends only on $c$.
\end{lemma}

\begin{proof} By Lemma \ref{lem:tail}, for
$y>(2\pi)^{-1/2}$, $x\leq e^{-y^2/2}$; hence $y\leq
\sqrt{2\ln(1/x)}$.
\begin{align*} \phi(-\Phi^{-1}(x) + c) &=
(2\pi)^{-1/2}\exp(-\tfrac{(y + c)^2}{2}) \\ &=
(2\pi)^{-1/2}\exp(-\tfrac{y^2}{2})\exp(-yc)\exp(-c^2/2) \\
&\leq 2y\Phi(-y)\exp(-yc) \qquad \text{ by Lemma
\ref{lem:tail} }\\ &\leq 2xy\exp(-yc).
\end{align*} The statement for $c>0$ follows from that
$\exp(-yc)\leq 1$ and $y\leq \sqrt{2\ln(1/x)}$.  For $c\leq
0$ the statement follows from that $y\leq C\exp(yC)$ for
some constant $C$.
\end{proof}
\begin{theorem}[Borell, see also \cite{Ledoux2001}, Theorem
7.1]\label{thm:borell} Let $G=(G_t)_{t\in T}$ be a centered
Gaussian process indexed by a countable set $T$ such that
$\sup_{t\in T} G_t < \infty$ almost surely. Then,
$\E(\sup_{t\in T}G_t)< \infty$ and for every $r>0$
\[ \P(\{\sup\limits_{t\in T} G_t \geq \E(\sup_{t\in T}G_t)
+r\})\leq e^{-r^2/2\sigma^2},
\] where $\sigma=\sup_{t\in T} (\E G_t^2)^{1/2}$.
\end{theorem}

For the following, let $B_t$ be a Brownian motion and
$F_t(x)$ the distribution function of $B_t$, which is
$\Phi(\frac{x}{\sqrt t})$.  Also for $1\leq t\leq 2$,
$0<\varepsilon<1/2$, set
\begin{align*} D&:=D(t, \varepsilon):= \sup_{t<s\leq t+\varepsilon} \tfrac{B_s
    - B_t}{\sqrt s},\\
  m&:=m(t,\varepsilon):=\E \sup_{t<s\leq
    t+\varepsilon} \tfrac{B_s - B_t}{\scriptstyle \sqrt s},\\
  m_0&:=\sup\{m(t,\varepsilon):1\leq t\leq 2,
  0<\varepsilon<1/2\}.
\end{align*} We use $C$ to denote a constant, which may vary
in each occurrence.
\begin{lemma}\label{lem:m} For $1\leq t\leq 2$,
$0<\varepsilon<1/2$
\[ m\leq 2(2/\pi)^{1/2}\varepsilon^{1/2}.
\]
\end{lemma}
\begin{proof} By the maximal inequality for Brownian motion,
\begin{align*}
  m&:=\E \sup_{t<s\leq t+\varepsilon}
\tfrac{B_s - B_t}{\scriptstyle \sqrt s}\\
&\leq\E\sup_{t<s\leq t+\varepsilon} \tfrac{|B_s -
B_t|}{\scriptstyle \sqrt{t}}\\
&\leq \E\varepsilon^{1/2}2|N(0,1)|\\ 
&\leq 2(2/\pi)^{1/2}\varepsilon^{1/2}. \qedhere
\end{align*}
\end{proof}
\begin{lemma}\label{lem:env} Let $d:=\E (\sup_{1\leq t\leq
2} \tfrac{B_t}{\sqrt t})$. Then, $d>0$ and
\[ \P(\inf_{1\leq t\leq 2} F_t(B_t) \leq x) \leq
(2\pi)^{1/2}\phi(-\Phi^{-1}(x)-d).\]
\end{lemma}
\begin{proof}
\begin{align*} \P(\inf_{1\leq t\leq 2}F_t(B_t)\leq x) &=\P(
\inf_{1\leq t\leq 2}\tfrac{B_t}{\sqrt t} \leq \Phi^{-1}(x))
\\ &=\P( \sup_{1\leq t\leq 2}\tfrac{-B_t}{\sqrt t} \geq
-\Phi^{-1}(x)) \\ &=\P(\sup_{1\leq t\leq
2}\tfrac{-B_t}{\sqrt t} \geq d -\Phi^{-1}(x)-d) \\ 
\intertext{which, by
Theorem~\ref{thm:borell} and for $x$ such that 
$-\Phi^{-1}(x)-d>0$, is}
&\leq \exp(-\tfrac{(-\Phi^{-1}(x)-d)^2}{2}) \\ 
&= (2\pi)^{1/2}\phi(-\Phi^{-1}(x)-d).
\end{align*}
Note that here $\sigma^2=\sup_{1\leq t\leq 2} \E (\tfrac{-B_t}{\sqrt t})^2=1$.
\end{proof}
\begin{lemma}\label{lem:envelope} Let
$w(x)=x^{-\alpha}L(x)$, $0<\alpha<1/2$ and $L(x)$ be a slowly
varying function (growing to infinity as $x\downarrow 0$).
Assume $w(x)$ is decreasing near $0$.  Then
\[ \lim_{\lambda\to \infty} \lambda^2\Pr(\sup_{1\leq t\leq 2}
w(F_t(B_t)) >\lambda) = 0.
\]
\end{lemma}
\begin{proof} Let $\lambda=w(x)$. Then,
by Lemma~\ref{lem:y} and Lemma~\ref{lem:env},
\begin{align*} \quad \lim_{\lambda\to \infty}
\lambda^2\Pr(\sup_{1\leq t\leq 2} w(F_t(B_t)) >\lambda) &=
\lim_{\lambda\to \infty} \lambda^2\Pr( w(\inf_{1\leq t\leq
2}F_t(B_t)) >\lambda) \\ &= \lim_{x\to 0} w(x)^2\Pr(
\inf_{1\leq t\leq 2}F_t(B_t)\leq x) \\ &\leq \lim_{x\to 0}
w(x)^2 (2\pi)^{1/2}\phi(-\Phi^{-1}(x)-d) \\ &\leq \lim_{x\to
0} x^{-2\alpha}L(x)^2(2\pi)^{1/2} CxL_C(x) \\ &=0. \qedhere
\end{align*}
\end{proof}
\begin{lemma}\label{lem:l} For $1\leq t\leq 2$,
$0<\varepsilon<1/2$, and $l>m$,
\[ \P (\tfrac{B_t}{\sqrt t }< l \leq \sup_{t<s\leq
t+\varepsilon} \tfrac{B_s}{\sqrt s}) \leq
C_t\varepsilon^{1/2}\phi(l-m)^{\tfrac{t+\varepsilon}{t+2\varepsilon}},
\] where $C_t$ is a constant depending only on $t$.  In
particular, if we let $C:=\sup_{1\leq t\leq 2} C_t$, and recall
$m_0:=\sup\{m(t,\varepsilon):1\leq t\leq 2,
0<\varepsilon<1/2\}$, then for $l>m_0$, we have $C<\infty$ and
\begin{equation}
  \label{eq:c} \P (\tfrac{B_t}{\sqrt t }< l \leq
\sup_{t<s\leq t+\varepsilon} \tfrac{B_s}{\sqrt s}) \leq
C\varepsilon^{1/2}\phi(l-m_0)^{\tfrac{t+\varepsilon}{t+2\varepsilon}}.
\end{equation}
\end{lemma}
\begin{proof} Since $\sigma^2:=\sup_{t<s\leq t+\varepsilon}
  \E (\tfrac{B_s - B_t}{\sqrt
    s})^2=\tfrac{\varepsilon}{t+\varepsilon}$, by Borell's
  concentration inequality Theorem [\ref{thm:borell}]
  (since the process $(B_s-B_t)/s^{1/2}$ is continuous in $s$,
 we can take supremum over a countable set in the definition of $D$)
  it follows
  that for $r>0$
\begin{equation}\label{eq:dev} \P(D>m+r)\leq
e^{-r^2(t+\varepsilon)/(2\varepsilon)}.
\end{equation} Hence, conditioning on $\tfrac{B_t}{\sqrt t}$,
\begin{align*} \P(\tfrac{B_t}{\sqrt t } < l \leq
\sup_{t<s\leq t+\varepsilon} \tfrac{B_s}{\sqrt s}) &\leq \P
(\tfrac{B_t}{\sqrt t } < l \leq \sup_{t<s\leq t+\varepsilon}
(\tfrac{B_s}{\sqrt s} - \tfrac{B_t}{\sqrt s } ) +
\sup_{t<s\leq t+\varepsilon}\tfrac{B_t}{\sqrt s })\\ &=
\E_{\tfrac{B_t}{\sqrt t}} \P( \tfrac{B_t}{\sqrt t } < l \leq
\sup_{t<s\leq t+\varepsilon} (\tfrac{B_s}{\sqrt s} -
\tfrac{B_t}{\sqrt s } ) + \sup_{t<s\leq
t+\varepsilon}\tfrac{B_t}{\sqrt s } |\tfrac{B_t}{\sqrt t})\\
\intertext{by independence of $\{B_s-B_t:s>t\}$ and
$B_t$ } &= \int_{-\infty}^{\infty} \P\big (( y< l \leq D +
\sup_{t<s\leq t+\varepsilon} \{ (t/s)^{1/2}y \} \big )
\tfrac{1}{\sqrt{2\pi}}e^{-y^2/2} \,dy \\ &= \int_0^{\infty}
\P\big( y< l \leq D + \sup_{t<s\leq t+\varepsilon} \{
(t/s)^{1/2}y \}\big ) \tfrac{1}{\sqrt{2\pi}}e^{-y^2/2} \,dy
\\ &\qquad + \int_{-\infty}^0 \P\big( y< l \leq D +
\sup_{t<s\leq t+\varepsilon} \{ (t/s)^{1/2}y \} \big )
\tfrac{1}{\sqrt{2\pi}}e^{-y^2/2} \,dy \\ &= I + II
\numberthis \label{eq:s}.
\end{align*} Note that
\begin{align*} &\sup_{t<s\leq t+\varepsilon} \{ (t/s)^{1/2}y
\}=y \quad \text{ for } y>0,\\ &\sup_{t<s\leq t+\varepsilon}
\{ (t/s)^{1/2}y \}=((t/(t+\varepsilon))^{1/2}y=:ay \quad
\text{ for } y\leq 0.
\end{align*} Therefore,
\begin{align*} I&= \int_0^{\infty} \P\big( y< l \leq D+y
\big ) \tfrac{1}{\sqrt{2\pi}}e^{-y^2/2} \,dy \\ &=
\int_0^{l-m} \P\big( y< l \leq D+y \big )
\tfrac{1}{\sqrt{2\pi}}e^{-y^2/2} \,dy + \int_{l-m}^{l}
\P\big( y< l \leq D+y \big )
\tfrac{1}{\sqrt{2\pi}}e^{-y^2/2} \,dy \\ &= \int_0^{l-m}
\P\big(D\geq l-y \big ) \tfrac{1}{\sqrt{2\pi}}e^{-y^2/2}
\,dy + \int_{l-m}^{l} \P\big( y< l \leq D+y \big )
\tfrac{1}{\sqrt{2\pi}}e^{-y^2/2} \,dy \\ \intertext{by
inequality~\eqref{eq:dev} for the first summand and noting
$r:=l-y-m>0$} &\leq \int_0^{l-m}
e^{-\tfrac{(l-y-m)^2(t+\varepsilon)}{2\varepsilon}
\tfrac{1}{\sqrt{2\pi}}e^{-y^2/2}} \,dy + \int_{l-m}^{l}
\P\big( y< l \leq D+y \big )
\tfrac{1}{\sqrt{2\pi}}e^{-y^2/2} \,dy\\ 
\intertext{by
completing the square in $y$ for the first summand} 
&\leq
(\tfrac{\varepsilon}{t+2\varepsilon})^{1/2}
e^{-\tfrac{(l-m)^2}{2}
\tfrac{t+\varepsilon}{t+2\varepsilon}} + m\phi(l-m)\\
\intertext{ bounding $m$ using Lemma~\ref{lem:m} }
&\leq
(\tfrac{\varepsilon}{t+2\varepsilon})^{1/2}(2\pi)^{1/2}
\phi(l-m)^{\tfrac{t+\varepsilon}{t+2\varepsilon}} +
2(2/\pi)^{1/2}\varepsilon^{1/2}\phi(l-m)  \numberthis \label{eq:s1}.
\end{align*} For $II$,
\begin{align*} II&= \int^0_{-\infty} \P\big( y< l \leq D+ay
\big ) \tfrac{1}{\sqrt{2\pi}}e^{-y^2/2} \,dy \\ &\leq
\int^0_{-\infty} \P\big(D\geq l-ay \big )
\tfrac{1}{\sqrt{2\pi}}e^{-y^2/2} \,dy \\ \intertext{by
equation~\eqref{eq:dev}} &\leq \int^0_{-\infty}
e^{-\tfrac{(l-ay-m)^2(t+\varepsilon)}{2\varepsilon}
\tfrac{1}{\sqrt{2\pi}}e^{-y^2/2}} \,dy \\ \intertext{by
completing the square in $y$} &\leq
(\tfrac{\varepsilon}{t+\varepsilon})^{1/2}
e^{-\tfrac{(l-m)^2}{2}}\\ &=
(\tfrac{\varepsilon}{t+\varepsilon})^{1/2} (2\pi)^{1/2}
\phi(l-m) \numberthis \label{eq:s2}.
\end{align*} Combining \eqref{eq:s}, \eqref{eq:s1},
and\eqref{eq:s2} completes the proof.
\end{proof}
\begin{lemma}\label{lem:-l} For $1\leq t\leq 2$,
  $0<\varepsilon\leq 1/2$, there is a universal constant $C$
  such that for $0<x<1/4$
  \[ \P(\tfrac{B_t}{\sqrt t} \leq \Phi^{-1}(x) <
  \sup_{t<s\leq t+\varepsilon} \tfrac{B_s}{\sqrt s}) \leq
  C\varepsilon^{1/2}(x\ln\tfrac{1}{x}).
\]
\end{lemma}
\begin{proof}
\begin{align*} &\quad \P(\tfrac{B_t}{\sqrt t} \leq
\Phi^{-1}(x) < \sup_{t<s\leq t+\varepsilon}
\tfrac{B_s}{\sqrt s}) \\ &= \P(\tfrac{B_t}{\sqrt t} \leq
\Phi^{-1}(x) < [\sup_{t<s\leq t+\varepsilon}
\tfrac{B_s}{\sqrt s} - \tfrac{B_t}{\sqrt s}] + \sup_{t<s\leq
t+\varepsilon}\tfrac{B_t}{\sqrt s}) \\ \intertext{letting
$D= \sup_{t<s\leq t+\varepsilon} \tfrac{B_s}{\sqrt s} -
\tfrac{B_t}{\sqrt s} $ and noting $B_t<0$ inside the
probability above} &\leq \E_D \P(\tfrac{B_t}{\sqrt t}\leq
\Phi^{-1}(x) \leq D+\tfrac{B_t}{\sqrt{t+\varepsilon}} |D )\\
\intertext{by independence of $ \{B_s-{B_t}:s>t\}$ and
$B_t$} &= \E_D \P(\tfrac{B_t}{\sqrt t}\leq \Phi^{-1}(x) \leq
\tfrac{B_t}{\sqrt{t+\varepsilon}} +D) \\ &= \E_D
\P((\tfrac{t+\varepsilon}{t})^{1/2}(\Phi^{-1}(x)-D)\leq
\tfrac{B_t}{\sqrt t} \leq \Phi^{-1}(x) )\\ \intertext{
bounding the density of $\tfrac{B_t}{\sqrt t}$ from above by
$\phi(\Phi^{-1}(x))$} &\leq \E_D
\phi(\Phi^{-1}(x))[(1-(\tfrac{t+\varepsilon}{t})^{1/2})
\Phi^{-1}(x) + (\tfrac{t+\varepsilon}{t})^{1/2}D] \\ &\leq
\phi(\Phi^{-1}(x))(-\Phi^{-1}(x))(\varepsilon/t) +
\phi(-\Phi^{-1}(x)) (\tfrac{t+\varepsilon}{t})^{1/2}\E_D D\\ 
&\leq C(x\ln\tfrac{1}{x})(\varepsilon/t) +
Cx(\ln\tfrac{1}{x})^{1/2}8\varepsilon^{1/2} \text{ by
Lemma~\ref{lem:y} and Lemma~\ref{lem:m}}\\ 
&\leq C\varepsilon^{1/2}(x\ln\tfrac{1}{x}). \qedhere
\end{align*}
\end{proof}
\begin{proposition}\label{prop:d1} For $1\leq t\leq 2$,
$0<\varepsilon\leq 1/2$, there is a universal constant $C$
such that for $0<x<1/4$
\[ \P(F_t(B_t) \leq x < \sup_{s:|s-t|\leq \varepsilon}
F_s(B_s)) \leq C\varepsilon^{1/2}(x\ln\tfrac{1}{x})
+C\varepsilon^{1/2}\phi(-\Phi^{-1}(x)-m_0)^{\tfrac{t}{t+\varepsilon}}.
\]
\end{proposition}
\begin{proof}
\begin{align*} 
  &\quad \P(F_t(B_t) \leq x <
  \sup_{\{s:|s-t|\leq \varepsilon\}} F_s(B_s)) \\
  &=\P(\Phi(\tfrac{B_t}{\sqrt t}) \leq x <
  \sup_{\{s:|s-t|\leq
    \varepsilon\}} \Phi(\tfrac{B_s}{\sqrt s})) \\
  &=\P(\tfrac{B_t}{\sqrt t} \leq \Phi^{-1}(x) <
  \sup_{\{s:|s-t|\leq \varepsilon\}} \tfrac{B_s}{\sqrt s}) \\
  &\leq \P(\tfrac{B_t}{\sqrt t} \leq \Phi^{-1}(x) <
  \sup_{t<s\leq t+\varepsilon} \tfrac{B_s}{\sqrt s}) +
  \P(\tfrac{B_t}{\sqrt t} \leq \Phi^{-1}(x) <
  \sup_{t-\varepsilon\leq s<t} \tfrac{B_s}{\sqrt s}) \\
  &= I + II.
\end{align*} By Lemma~\ref{lem:-l},
\begin{equation}
  \label{eq:d1_1} I\leq
C\varepsilon^{1/2}(x\ln\tfrac{1}{x}).
\end{equation} Now we consider $II$.
\begin{align*} II&=\P(\tfrac{B_t}{\sqrt t} \leq
  \Phi^{-1}(x)
  < \sup_{t-\varepsilon\leq s<t} \tfrac{B_s}{\sqrt s})\\
  &=\P(\tfrac{B_{t-\varepsilon}}{\sqrt{t-\varepsilon}}\leq
  \Phi^{-1}(x),\ \tfrac{B_t}{\sqrt t} \leq \Phi^{-1}(x) <
  \sup_{t-\varepsilon\leq s<t} \tfrac{B_s}{\sqrt s})\\
  &\qquad
  +\P(\tfrac{B_{t-\varepsilon}}{\sqrt{t-\varepsilon}} >
  \Phi^{-1}(x),\ \tfrac{B_t}{\sqrt t} \leq \Phi^{-1}(x) <
  \sup_{t-\varepsilon\leq s<t} \tfrac{B_s}{\sqrt s})\\
  &\leq \P(\tfrac{B_{t-\varepsilon}}{\sqrt{t-\varepsilon}}
  \leq \Phi^{-1}(x) < \sup_{t-\varepsilon<s\leq t}
  \tfrac{B_s}{\sqrt s}) + \P(\tfrac{B_t}{\sqrt t}\leq
  \Phi^{-1}(x) <
  \tfrac{B_{t-\varepsilon}}{\sqrt{t-\varepsilon}})\\
  &\leq \P(\tfrac{B_{t-\varepsilon}}{\sqrt{t-\varepsilon}}
  \leq \Phi^{-1}(x) < \sup_{t-\varepsilon<s\leq t}
  \tfrac{B_s}{\sqrt s}) +
  \P(\tfrac{B_{t-\varepsilon}}{\sqrt{t-\varepsilon}}\leq
  -\Phi^{-1}(x) < \tfrac{B_t}{\sqrt t})\\
  &\leq C\varepsilon^{1/2}(x\ln\tfrac{1}{x}) +
  C\varepsilon^{1/2}\phi(-\Phi^{-1}(x)-m_0)^{\tfrac{t}{t+\varepsilon}}
  \text{\quad by Lemmas~\ref{lem:-l} and
    \ref{lem:l}}. \numberthis\label{eq:II}
\end{align*}
\end{proof}
\begin{proposition}\label{prop:d2} 
  For $1\leq t\leq 2$, $0<\varepsilon\leq 1/2$, there is a
  universal constant $C$ such that for $0<x<1/4$
  \[ \P (\inf_{\{s: |s-t|\leq \varepsilon\}} F_s(B_s)\leq x
  <F_t(B_t)) \leq C\varepsilon^{1/2} \phi(-\Phi^{-1}(x)
  -m_0)^{\tfrac{t}{t+\varepsilon}} +
  C\varepsilon^{1/2}(x\ln\tfrac{1}{x}).
\]
\end{proposition}
\begin{proof} First we consider the case $\{s>t: |s-t|\leq
\varepsilon\}$.  Let $D= \sup_{t<s\leq t+\varepsilon}
\tfrac{B_s}{\sqrt s} - \tfrac{B_t}{\sqrt s} $.
\begin{align*} 
  &\P (\inf_{t<s\leq t+\varepsilon} F_s(B_s)
  \leq x <F_t(B_t)) \\
  &= \P(\inf_{t<s\leq t+\varepsilon} \tfrac{B_s}{\sqrt s}
  \leq \Phi^{-1}(x)< \tfrac{B_t}{\sqrt t
  }) \\
  &= \P(\tfrac{B_t}{\sqrt t } < -\Phi^{-1}(x)\leq
  \sup_{t<s\leq t+\varepsilon} \tfrac{B_s}{\sqrt s}) \\
  &\leq C \varepsilon^{1/2} \phi(-\Phi(x)
  -m_0)^{\tfrac{t+\varepsilon}{t+2\varepsilon}}\quad \text{by
    Lemma~\ref{lem:l}}.
\end{align*} For the the case $\{s<t: |s-t|\leq
\varepsilon\}$,
\begin{align*} &\P (\inf_{t-\varepsilon\leq s<t} F_s(B_s)
\leq x <F_t(B_t)) \\
&= \P(\tfrac{B_t}{\sqrt t } <
-\Phi^{-1}(x)\leq \sup_{t-\varepsilon\leq s<t}
\tfrac{B_s}{\sqrt s}) \\ &=
\P(\tfrac{B_{t-\varepsilon}}{\sqrt{t-\varepsilon}} <
-\Phi^{-1}(x), \tfrac{B_t}{\sqrt t } < -\Phi^{-1}(x)\leq
\sup_{t-\varepsilon\leq s<t} \tfrac{B_s}{\sqrt s}) \\ &\quad
+ \P(\tfrac{B_{t-\varepsilon}}{\sqrt{t-\varepsilon}} \geq
-\Phi^{-1}(x), \tfrac{B_t}{\sqrt t } < -\Phi^{-1}(x)\leq
\sup_{t-\varepsilon\leq s<t} \tfrac{B_s}{\sqrt s}) \\ &=
\P(\tfrac{B_{t-\varepsilon}}{\sqrt{t-\varepsilon}} <
-\Phi^{-1}(x) \leq \sup_{t-\varepsilon\leq s<t}
\tfrac{B_s}{\sqrt s}) + \P(\tfrac{B_t}{\sqrt t }<
-\Phi^{-1}(x) \leq
\tfrac{B_{t-\varepsilon}}{\sqrt{t-\varepsilon}})\\ &=
\P(\tfrac{B_{t-\varepsilon}}{\sqrt{t-\varepsilon}} <
-\Phi^{-1}(x) \leq \sup_{t-\varepsilon\leq s<t}
\tfrac{B_s}{\sqrt s}) + \P(
\tfrac{B_{t-\varepsilon}}{\sqrt{t-\varepsilon}}\leq
\Phi^{-1}(x)<\tfrac{B_t}{\sqrt t } )\\
&\leq C
\varepsilon^{1/2} \phi(-\Phi(x)
-m_0)^{\tfrac{t}{t+\varepsilon}}
+C\varepsilon^{1/2}(x\ln\tfrac{1}{x}) \text{ by Lemmas~\ref{lem:l}
and \ref{lem:-l} }. \qedhere
\end{align*}
\end{proof}
\begin{proof}[Proof of Theorem~\ref{thm:ex}] Let
  $0<\varepsilon<1/2 $ and $1\leq t\leq 2 $.  Choose
  $\theta>4$ big enough such that
  $\tfrac{t}{t+\varepsilon^\theta}>2\alpha$ uniformly in $t$
  and $\varepsilon$.  Let $\rho(s,t)=|s-t|^{1/\theta}$. Then
  $\rho(s,t)$ is a continuous Gaussian metric on $[0,1]$
  (indeed it is the $L_2$ distance of the fractional
  Brownian motion with Hurst index ${1/\theta}$).  By
  Lemmas~\ref{lem:sen}, \ref{lem:L}, and \ref{lem:y}, it
  follows that for $0<x<1/4$ (for $1/4\leq x\leq 1/2$, the proof is trivial as $w(\cdot)$ is uniformly bounded on it)
\[ \phi(-\Phi^{-1}(x)
-m_0)^{\tfrac{t}{t+\varepsilon^\theta}} \leq
[CxL_C(x)]^{\tfrac{t}{t+\varepsilon^\theta}} \leq
Cx^{2\alpha}/L(x)=\frac{C}{w(x)^2}.
\] Hence Propositions \ref{prop:d1} and \ref{prop:d2}
verify the WL-condition in Theorem~\ref{thm:main1} and
Lemma~\ref{lem:envelope} verifies the envelope function
condition therein. Hence by part ($i$) of Proposition~\ref{prop:X2Y} and
noting the distribution functions $F_t$ of $B_t$ are
strictly increasing, we conclude the proof.
\end{proof}

\section*{Appendix}
In this appendix, we give the proof of Lemma~\ref{lem:extension}.
\begin{proof}[Proof of Lemma~\ref{lem:extension}]
  We denote the restricted process $\{G(t):t\in T_0\}$ by $G_0$. Then
  almost surely its sample paths are uniformly continuous on $T_0$.
  Each sample path can be extended to a uniformly continuous sample
  path on $T$. Indeed, if we let $G_0(\omega)$ be a sample path and $t\in
  T$, then there is a sequence, say $(t_m)\subset T_0$, such that
  $d_G(t_m, t)\rightarrow 0$ as $m\to \infty $ and define $\tilde
  G(t)(\omega):=\lim_{m\to \infty} G(t_m)(\omega)$.  It's easy to see
  it's well defined and is uniformly $d_G$ continuous on $T$. Moreover, in view of
  its characteristic function, $\tilde G(t)$ is normal. Let
  $\tilde \rho $ be the covariance of $\tilde G$.  It remains to show
  $\rho=\tilde \rho $.  But that $\rho$ and $\tilde \rho$ coincide on
  $T_0\times T_0$
  implies they
  coincide on $T\times T$. Indeed, for any $s,t\in T$, we can find
  a sequences $(s_m)$ and $(t_m)$ in $T_0$, such that
  $d_G(s_m, s)\rightarrow 0$ and
  $d_G(s_t, t)\rightarrow 0$.
  Then $|\rho(s,t) - \rho(s_m, t_m)|\rightarrow 0$.
\end{proof}
\begin{ack}
The author would like to thank Prof. J. Zinn for helpful suggestions and the referees
for careful reading of the paper and the comments, which have led to many improvements.
\end{ack}
\bibliographystyle{plain}

\def\cprime{$'$}

\end{document}